\documentclass[graybox]{svmult}

\usepackage{type1cm}        % activate if the above 3 fonts are
\usepackage{makeidx}         % allows index generation
\usepackage{graphicx}        % standard LaTeX graphics tool
\usepackage{multicol}        % used for the two-column index
\usepackage[bottom]{footmisc}% places footnotes at page bottom

\usepackage{newtxtext}       % 
\usepackage{newtxmath}       % selects Times Roman as basic font
\usepackage{amsmath,amssymb,mathrsfs, mathtools}
\DeclarePairedDelimiter\ceil{\lceil}{\rceil}
\DeclarePairedDelimiter\floor{\lfloor}{\rfloor}

\makeindex             % used for the subject index
\begin{document}
	
	\title*{On the Scalability of the Parallel Schwarz Method in One-Dimension}
	
	\author{Gabriele Ciaramella, Muhammad Hassan and Benjamin Stamm}
	\authorrunning{G. Ciaramella, M. Hassan and B. Stamm}
	\institute{Gabriele Ciaramella \at Universit\"at Konstanz \email{gabriele.ciaramella@uni-konstanz.de}
		\and Muhammad Hassan \at RWTH Aachen University \email{hassan@mathcces.rwth-aachen.de}
		\and Benjamin Stamm \at RWTH Aachen University \email{stamm@mathcces.rwth-aachen.de}}
	\maketitle
	
	\abstract*{In contrast with classical Schwarz theory, recent results in computational chemistry have shown that for special domain geometries, the one-level parallel Schwarz method can be scalable. This property is not true in general, and the issue of quantifying the lack of scalability remains an open problem. Even though heuristic explanations are given in the literature, a rigorous and systematic analysis is still missing. In this short manuscript, we provide a first rigorous result that precisely quantifies the lack of scalability of the classical one-level parallel Schwarz method for the solution to the one-dimensional Laplace equation. Our analysis technique provides a possible roadmap for a systematic extension to more realistic problems in higher dimensions.}
	
	\section{Introduction and main results}
	An algorithm is said to be weakly scalable if it can solve progressively larger problems with an increasing number of processors in a fixed amount of time. According to classical Schwarz theory, the parallel Schwarz method (PSM) is not scalable (see, e.g., \cite{Hassan_mini_3_CiaramellaGander4, Hassan_mini_3_ToselliWidlund}).
	Recent results in computational chemistry, however, have shed more light on the scalability of the PSM: surprisingly, in contrast with classical Schwarz theory, the authors in \cite{Hassan_mini_3_Stamm3} provide numerical evidence that in some cases the one-level PSM converges to a given tolerance within the same number of iterations independently of the number $N$ of subdomains. This behaviour is observed if fixed-sized subdomains form a ``chain-like'' domain {such that the intersection of the boundary of each subdomain with the boundary of the global domain is non-empty}. This result was subsequently rigorously proved in \cite{Hassan_mini_3_CiaramellaGander,Hassan_mini_3_CiaramellaGander2,Hassan_mini_3_CiaramellaGander3} for the PSM and in \cite{Hassan_mini_3_CiaramellaGander4} for other one-level methods. On the other hand, this weak scalability is lost if the fixed-sized subdomains form a ``globular-type'' domain $\Omega$, where the boundaries of many subdomains lie in the interior of $\Omega$. The following question therefore arises: is it possible to quantify the lack of scalability of the PSM for cases where {individual subdomains are entirely embedded inside the global domain}? To do so, for increasing $N$ one would need to estimate the number of iterations necessary to achieve a given tolerance.
	
	Some isolated results in this direction do exist in the literature. For instance, in \cite{Hassan_mini_3_CiaramellaGander4} a heuristic argument is used to explain why in the case of the PSM for the solution of a 1D Laplace problem an unfortunate initialisation leads to a contraction in the infinity norm being observed only after a number of iterations proportional to $N$. Similarly, for a special choice of overlapping subdomains, an elegant result can be found in \cite{Hassan_mini_3_Gander2002} where the authors prove that the Schwarz waveform-relaxation, for the solution of the heat equation, contracts at most every $m+2$ iterations with $m$ being an integer representing the maximum distance of the subdomains from the boundary. Nevertheless, the literature does not contain a comprehensive study of this problem for a general decomposition. Furthermore, existing results unfortunately do not provide a systematic approach to build on and extend in order to cover more general settings. Our goal therefore has been to develop a \emph{framework} that can be applied to a broad class of overlapping subdomains, in multiple dimensions and containing sub-domains with an arbitrary number and type (double, triple, quadruple and so on) of intersections, such as for molecular domains in computational chemistry \cite{Hassan_mini_3_Stamm3}. 
	
	Of course, tackling this problem in a completely general setting is a daunting task. The purpose of the current article is to develop such a new framework and apply it to the PSM for the solution of a 1D Laplace problem as a 'toy' problem. The key elements of our proposed framework are
	\begin{itemize}
		\item the identification of an adequate norm for studying the properties of the Schwarz operator,
		\item the maximum principle,
		\item and the idea of tracking the propagation of the contraction towards the interior of
		the global domain $\Omega$.
	\end{itemize}
	
	Our expectation is that a framework based on these ingredients can then be systematically extended to more general decompositions of the domain {which can be quite complex in two and three dimensions}. We emphasise that most (but not all) of the results we prove are either known or intuitively clear. Our true contribution is the new \emph{analysis technique} that we introduce. On the one hand, this technique results in a deeper understanding of the method and leads to a sharper description of the contraction behaviour. On the other hand, the tools developed in this article also suggest a systematic roadmap to extend our results to more realistic problems in higher dimensions. {In principle, this can be done by carefully tracking the propagation of the contraction through the subdomains as attempted in this manuscript. Note however that the contraction behaviour for a 1D reaction-diffusion equation is completely different from that of the 1D Laplace equation. This can be proved as shown in \cite{Hassan_mini_3_CiaramellaGander4, Hassan_mini_3_CiaramellaGander} }.

	We first state our main results. We consider the Laplace equation in one-dimension. Let $L>0$, we must find a function $e \colon [0, L] \rightarrow \mathbb{R}$ that solves
	\begin{equation}\label{Hassan_mini_3_eq:1}
	(e)^{\prime \prime}(x)=0 \quad \forall ~x \in (0,L), \qquad e(0)=0, \qquad e(L)=0. 
	\end{equation}
	Clearly \eqref{Hassan_mini_3_eq:1} represents an error equation whose solution is trivially $e=0$.
	In order to apply the PSM to solve \eqref{Hassan_mini_3_eq:1},
	we consider a decomposition $\Omega = \cup_{j=1}^N \Omega_j$, where $\Omega_j:=(a_j,b_j)$ 
	with $a_1=0$, $b_1=a_1 +\ell$ and $a_{j+1}= j(\ell-\delta)$, $b_{j+1} = a_{j+1}+\ell$ for $j=1,\dots,N-1$.
	Here, $\ell>0$ is the length of each subdomain, $\delta>0$ the overlap, and it holds that $L=N\ell-(N-1)\delta$. Now, let $e_0 \colon \Omega \rightarrow \mathbb{R}$ be some initialization.
	The PSM defines the sequences $\{e^n_j\}_{n \in \mathbb{N}}$ by solving for each $n \in \mathbb{N}$ the sub-problems
	\begin{equation}\label{Hassan_mini_3_eq:3}
	(e^n_j)^{\prime \prime}(x) =0 \quad \forall ~x \in (a_j, b_j), \quad e^n_j(a_j)=e^{n-1}_{j-1}(a_j), \quad e^n_j(b_j)=e^{n-1}_{j+1}(b_j), 
	\end{equation}
	for each $j=2,\dots,N-1$ and 
	\begin{align*}
	(e^n_1)^{\prime \prime}(x) &=0 \quad \forall ~x \in (a_1, b_1), \qquad e^n_1(a_1)=0, &e^n_1(b_1)&=e^{n-1}_{2}(b_1), \\
	(e^n_N)^{\prime \prime}(x) &=0 ~~ \forall ~x \in (a_N, b_N), \qquad e^n_N(a_N)=e^{n-1}_{N-1}(a_N), &e^n_N(b_N)&=0. 
	\end{align*}
	
	Solving \eqref{Hassan_mini_3_eq:1} and \eqref{Hassan_mini_3_eq:3} 
	and defining for each $n \in \mathbb{N}$
	the vector $\bold{e}^n \in \mathbb{R}^{2N}$ as
	$$\bold{e}^n := \begin{bmatrix}
	e^n_1(a_1) &
	e^n_1(a_2) & 
	e^n_2(b_1) &
	e^n_2(a_3) &
	\cdots &
	e^n_j(b_{j-1})&
	e^n_j(a_{j+1})&
	\cdots &
	e^n_N(b_{N-1}) &
	e^n_N(b_{N})
	\end{bmatrix}^\top,$$ 
	it is possible to write the PSM iterations as
	$\bold{e}^{n+1}= T \bold{e}^n$.
	Here $T \in \mathbb{R}^{2N \times 2N}$ is a non-negative ($T_{j,k}\geq 0$), non-symmetric block tridiagonal matrix:
	\begin{equation*}
	T = \begin{bmatrix}
	0 & \widetilde{ T_2}&  0 &\cdots & 0 & 0 &0\\
	T_1 & 0 &  T_2 &\cdots & 0 & 0 & 0 \\
	\ddots & \quad\ddots & \quad\ddots & \quad\ddots & \ddots & \ddots & \ddots\\
	0 & 0&  0 &\cdots &  T_1 & 0 & T_2\\
	0 & 0&  0 &\cdots & 0 & \widetilde{ T_1} &0
	\end{bmatrix},
	\hspace*{-30mm}
	\begin{split}
	T_1 &= \begin{bmatrix}
	0 & 1-\frac{\delta}{\ell}\\
	0 & \frac{\delta}{\ell}
	\end{bmatrix}, 
	\; 
	\widetilde{ T_1} = \begin{bmatrix}
	0 & 1-\frac{\delta}{\ell}\\
	0 & 0
	\end{bmatrix}, \\
	\; 
	T_2 &= \begin{bmatrix}
	\frac{\delta}{\ell}& 0\\
	1-\frac{\delta}{\ell}& 0
	\end{bmatrix},
	\; 
	\widetilde{ T_2} = \begin{bmatrix}
	0 & 0\\
	1-\frac{\delta}{\ell} & 0
	\end{bmatrix}.
	\end{split}
	\end{equation*}
	We denote by $\| \cdot \|$ the usual infinity norm and the corresponding induced matrix norm. 
	Our goal is to analyze the convergence properties of the PSM sequence $\{e_j^n\}_{n \in \mathbb{N}}$ with respect to $\| \cdot \|$.
	Hence, we must study the properties of the matrix $ T$.
	Our main results are summarized in the following theorem.
	\begin{theorem}\label{Hassan_mini_3_thm:Main}
		Let $N \in \mathbb{N}$ be the number of subdomains in $\Omega$ and $\bold{1}_{N} \in \mathbb{R}^{2N}$ the vector whose elements are all equal to $1$. Then $\Vert  T^n \Vert = \Vert  T^n \bold{1}_{N}\Vert \leq 1$ for any $n\in\mathbb{N}$ and
		\begin{itemize}\itemsep0em
			\item[{\rm(a)}] $\Vert  T^{\ceil{\frac{N}{2}}}\Vert < 1$ and hence $\rho(T) < 1$, where $\rho( T)$ is the spectral radius of $ T$.
			\item[{\rm(b)}] $\Vert  T^{n+1} \Vert < \Vert  T^{n} \Vert$ if $N$ is even, and $\Vert  T^{n+2} \Vert < \Vert  T^{n} \Vert$ if $N$ is odd, for $n \geq \ceil*{\frac{N}{2}}$.
		\end{itemize}
	\end{theorem}
	
	Theorem \ref{Hassan_mini_3_thm:Main} states clearly that the PSM converges. {Moreover, it identifies $\bold{1}_{N}$ as the unit vector that maximises the $\ell_{\infty}$ operator norm of the iteration matrices $T^n, n \in \mathbb{N}$}.  This fact is then used to prove Theorem \ref{Hassan_mini_3_thm:Main} {\rm(b)}: if initialized with $\bold{1}_{N}$, after $\ceil{\frac{N}{2}}$ iterations the PSM sequence contracts in the infinity norm at \emph{every iteration} if $N$ is even, or \emph{every second iteration} if $N$ is odd. Although proven for a 1-D problem, this result is much sharper than the one found in \cite{Hassan_mini_3_Gander2002}, {which states that the PSM sequence contracts in the infinity norm at most every $\ceil*{\frac{N}{2}}$ iterations}. To prove Theorem \ref{Hassan_mini_3_thm:Main} {\rm(b)} we use Lemmas \ref{Hassan_mini_3_lem:5} and \ref{Hassan_mini_3_cor:2}. These two technical results characterize precisely the shape of the vector $\bold{e}^{n}=  T^n \bold{1}_N$ at every iteration $n$ and clearly show how the contraction propagates from the two points of $\partial \Omega$ towards the subdomains in the middle of $\Omega$. 
	
	We prove Theorem \ref{Hassan_mini_3_thm:Main} in the following sections. In particular, in Section \ref{Hassan_mini_3_subs:1} we prove
	first that $\Vert  T^n \Vert= \Vert  T^n \bold{1}_N \Vert$ and then Theorem \ref{Hassan_mini_3_thm:Main} (a).
	In Section \ref{Hassan_mini_3_subs:2} we prove Theorem \ref{Hassan_mini_3_thm:Main} (b). Notice that using Theorem \ref{Hassan_mini_3_thm:Main}, one could also estimate the spectral radius of $ T$ as $$\rho( T) \leq \Vert  T^{\ceil{\frac{N}{2}}}\Vert^{1/\ceil{\frac{N}{2}}} =\Bigl[ 1 - \Bigl(\frac{\delta}{L}\Bigl)^{\ceil{\frac{N}{2}}} \Bigl]^{1/\ceil{\frac{N}{2}}}.$$ This can be proved by a direct calculation involving geometric arguments. However, since this is a quite conservative bound, we will not prove this result in this short article.
	
	\vspace{-0.5cm}
	\section{Proof of Theorem 1 $(\rm{a})$}\label{Hassan_mini_3_subs:1}
	
	In what follows, we use
	$P =\begin{bmatrix}
	0& 1\\
	1 & 0
	\end{bmatrix}$.
	Moreover, let $m \in \mathbb{N}$ and $\bold{v}, \bold{w} \in \mathbb{R}^m$, then 
	$\bold{v} < \bold{w}$ (resp. $\bold{v} > \bold{w}$) means that $v_i < w_i$ (resp. $v_i > w_i$) for each $i \in \{1, \ldots, m\}$.

	\begin{lemma}\label{Hassan_mini_3_lemma:1}
		For all $n \in \mathbb{N}$ it holds that $\Vert  T^n \Vert= \Vert  T^n \bold{1}_N \Vert$.
		%Let $T \in \mathbb{R}^{2N \times 2N}$ be a non-negative matrix.
		%Then $\Vert T \Vert= \Vert T \bold{1}_N \Vert$.
	\end{lemma}
	\begin{proof}
		Let $\bold{w}=  T^n \bold{1}_N$. Then for each $i \in \{1, \ldots, 2N\}$ it holds that 
		$w_i= \sum_{j=1}^{2N} ( T^n)_{ij}$.
		Since $ T$ is non-negative, $ T^n$ is also non-negative for any $n \in \mathbb{N}$
		and it holds that
		\begin{align*}
		\Vert  T^n \Vert
		= \max_{i=1, \ldots, 2N}\sum_{j=1}^{2N} (T^n)_{ij}
		=\max_{i=1, \ldots, 2N} \left\vert\sum_{j=1}^{2N} (T^n)_{ij}\right\vert
		= \max_{i=1, \ldots, 2N} \vert w_i\vert=\Vert  T^n \bold{1}_N\Vert.
		\end{align*}
	\end{proof}
	
	\noindent
	Next, let $a, b, c, d \in [0, 1)$ be real numbers such that $a<b\leq c<d$. Direct calculations show that the matrices $ T_1$ and ${ T_2}$ satisfy the following relations:
	\begin{align}
	%\widetilde{ T_2} \bold{1}_1&= \begin{bmatrix}0\\ 1-\frac{\delta}{\ell}\end{bmatrix}, \hspace{2mm}\qquad %\widetilde{ T_1} \bold{1}_1= \begin{bmatrix}1-\frac{\delta}{\ell}\\ 0\end{bmatrix}, \label{eq:1.4}\\
	%\widetilde{ T_2} \begin{bmatrix}a\\ b\end{bmatrix}&= \begin{bmatrix}0\\ (1-\frac{\delta}{\ell})a\end{bmatrix}, \qquad \widetilde{ T_1} \begin{bmatrix}b\\ a\end{bmatrix}= \begin{bmatrix}(1-\frac{\delta}{\ell})a \label{eq:1.5}\\ 0\end{bmatrix},\\[0.5em]
	b\,{\bold{1}_1}< T_1\begin{bmatrix}a\\ b\end{bmatrix}+  T_2\bold{1}_1&= \begin{bmatrix}(1-\frac{\delta}{\ell})b+\frac{\delta}{\ell}\\ \frac{\delta}{\ell}b + 1-\frac{\delta}{\ell}\end{bmatrix} < {\bold{1}_1},\label{Hassan_mini_3_eq:1.7}\\[0.5em]
	b\,{\bold{1}_1}\leq  T_1\begin{bmatrix}a\\ b\end{bmatrix}+  T_2\begin{bmatrix}c\\ d\end{bmatrix}&= \begin{bmatrix}(1-\frac{\delta}{\ell})b+\frac{\delta}{\ell}c\\ \frac{\delta}{\ell}b + (1-\frac{\delta}{\ell})c\end{bmatrix}\leq c\,{\bold{1}_1},\label{Hassan_mini_3_eq:1.9}\\[0.5em]
	T_1\bold{1}_1+  T_2\bold{1}_1= \bold{1}_1, \qquad  T_1\begin{bmatrix}a\\ b\end{bmatrix}&+  T_2\begin{bmatrix}c\\ d\end{bmatrix}= P\, \left(T_1\begin{bmatrix}d\\ c\end{bmatrix}+  T_2\begin{bmatrix}b\\ a\end{bmatrix}\right),\label{Hassan_mini_3_eq:1.11}
	\end{align}
	where the equality in \eqref{Hassan_mini_3_eq:1.9} holds if and only if $b=c$.
	
	\begin{definition}
		Let $n \in \left\{1, \ldots, \ceil*{\frac{N}{2}}\right\}$ be a natural number, we define $V^n \subset \mathbb{R}^{2N}$ as
		\begin{align*}
		V^n :=\left\{\bold{v}:=(\bold{v}_1, \ldots, \bold{v}_N) \colon \begin{cases}
		\bold{v}_j < \bold{1}_1 &\; \text{if } j \in \{ 1,\dots,n\} \cup \{ N+1-n,\dots, N\} \\
		\bold{v}_j =\bold{1}_1 &\; \text{otherwise}
		\end{cases}\right\}.
		\end{align*}
	\end{definition}
	
	We now state and prove the main result that will lead directly to Theorem \ref{Hassan_mini_3_thm:Main} ${\rm (a)}$.
	
	\begin{lemma}\label{Hassan_mini_3_lem:4}
		Let $n \in \left\{1, \ldots, \ceil*{\frac{N}{2}}\right\}$ be a natural number and let $\bold{w}= T^n \bold{1}_N$. Then it holds that $ T^n \bold{1}_N \in V^n$, and for all $j \in \{1, \ldots, N\}$ it holds that $\bold{w}_j = P \, \bold{w}_{N+1-j}$.
	\end{lemma}
	\begin{proof}
		We proceed by induction on the iteration index $n$. Let $n=1$ and $\bold{w} =  T\bold{1}_N$. Then by definition of the iteration matrix $T$ and the matrices $\widetilde{ T_1}, \widetilde{ T_2}$ it holds that
		\begin{align*}
		\bold{w}_1= \widetilde{ T_2} \bold{1}_1=\begin{bmatrix} 0\\ 1-\frac{\delta}{\ell}\end{bmatrix} < \bold{1}_1,\qquad
		\bold{w}_N= \widetilde{ T_1} \bold{1}_1= \begin{bmatrix} 1-\frac{\delta}{\ell}\\ 0\end{bmatrix} < \bold{1}_1,
		\end{align*}
		so that $\bold{w}_1= P \, \bold{w}_N$. Furthermore, by \eqref{Hassan_mini_3_eq:1.11} it holds that $\bold{w}_j=  T_1\bold{1}_1+ T_2\bold{1}_1=\bold{1}_1$ for all $j \in \{2, \ldots, N-1\}$, and thus $\bold{w}_j= P \,\bold{w}_{N+1-j}$. Hence, Lemma \ref{Hassan_mini_3_lem:4} holds for $n=1$. 
		
		%Next, let $n=2$, $\bold{u}=  T\bold{1}_N$ and $\bold{w}= T^2 \bold{1}_N$. Using the definition of the matrices $\widetilde{ T_1}, \widetilde{ T_2}$ we again obtain that $\bold{w}_1,\bold{w}_N <\bold{1}_1$ and $\bold{w}_1= P \, \bold{w}_N$. Next, using \eqref{eq:1.11} we obtain that
		%\begin{align*}
		%\bold{w}_2&=  T_1 \bold{u}_{1} +  T_2 \bold{u}_{3}=  T_1\begin{bmatrix} 0\\ 1-\frac{\delta}{\ell}\end{bmatrix}+ T_2\bold{1}_1\\
		%&=P \, \left(T_1\bold{1}_1+ T_2\begin{bmatrix} 1-\frac{\delta}{\ell}\\ 0\end{bmatrix}\right)= P \, \left( T_1 \bold{u}_{N-2} +  T_2 \bold{u}_{N}\right)= P \, \bold{w}_{N-1}.
		%\end{align*}
		%
		%In addition, using \eqref{eq:1.7}, we obtain that $\bold{w}_2=   T_1\begin{bmatrix} 0\\ 1-\frac{\delta}{\ell}\end{bmatrix}+ T_2\bold{1}_1<\bold{1}_1$.
		%Finally, \eqref{eq:1.11} implies that for all $j \in \{3, \ldots, N+1-3\}$ it holds that
		%\begin{align*}
		%\bold{w}_j&=  T_1 \bold{u}_{j-1} +  T_2 \bold{u}_{j+1}=  T_1\bold{1}_1+ T_2\bold{1}_1=\bold{1}_1.
		%\end{align*}
		%Thus, Lemma \ref{lem:4} also holds for $n=2$. 
		
		Assume now that Lemma \ref{Hassan_mini_3_lem:4} holds for some $n \in \left\{1, \ldots, \ceil*{\frac{N}{2}}-1\right\}$. We must show that Lemma \ref{Hassan_mini_3_lem:4} also holds for $n+1$. 
		Let $\bold{u} =  T^n \bold{1}_N$ and let $\bold{w}=  T^{n+1}\bold{1}_N$. %By the induction hypothesis, we know that $\bold{u}_2 = P \, \bold{u}_{N-1}$ and that $\bold{u}_2 < \bold{1}_1$. It follows that
		%\begin{align*}
		%\bold{w}_1 &= \widetilde{ T_2} \bold{u}_2 = \begin{bmatrix} 0\\ (1-\frac{\delta}{\ell}) (\bold{u}_2)_1\end{bmatrix}=P \, \begin{bmatrix} (1-\frac{\delta}{\ell}) (\bold{u}_{N-1})_2\\ 0\end{bmatrix}=P \, \widetilde{ T_1} \bold{u}_{N-1} =P \, \bold{w}_N < \bold{1}_1.
		%\end{align*}
		We proceed in three parts. First, we prove that the result holds in the case $n \geq 2$ for indices $j \in \{1, \ldots, n-1\}$, then we prove it for the index $j=n$ and finally for the index $j=n+1$. Note that it is necessary to proceed in these three steps since in each of these cases $\bold{w}_j$ depends on $\bold{u}_{j-1}$ and $\bold{u}_{j+1}$ which take different values depending on the index $j$.
		
		\begin{enumerate}
			\item $n \geq 2$ and $j \in \{1, \ldots, n-1\}$: Assume first that $j=1$. It follows from the induction hypothesis that $\bold{u}_2=P \, \bold{u}_{N-1} < \bold{1}_1$. A direct calculation similar to the one for the base case $n=1$ reveals that $\bold{w}_1= P \, \bold{w}_N < \bold{1}_1$. Now assume $j \neq 1$. It follows by the induction hypothesis that $\bold{u}_{j-1}=P \, \bold{u}_{N+2-j}$,
			and $P \,\bold{u}_{j+1}=\bold{u}_{N-j}$, and $\bold{u}_{j-1}, \bold{u}_{j+1} < \bold{1}_1$. We therefore obtain from \eqref{Hassan_mini_3_eq:1.11} that
			\begin{equation*}
			\medmuskip=-1mu
			\thinmuskip=-1mu
			\thickmuskip=-1mu
			\nulldelimiterspace=0.9pt
			\scriptspace=0.9pt    
			\arraycolsep0.9em
			\bold{w}_j= T_1 \bold{u}_{j-1}+  T_2 \bold{u}_{j+1}= P \,\left(T_1 P \, \bold{u}_{j+1}+  T_2 P \, \bold{u}_{j-1}\right)
			=P \,\left(T_1 \bold{u}_{N-j}+  T_2 \bold{u}_{N-j+2}\right)=P \, \bold{w}_{N-j+1} < \bold{1}_1.
			\end{equation*} 
			
			\item $j=n$: The induction hypothesis implies that $\bold{u}_{n-1}= P \, \bold{u}_{N+2-n}$, $\bold{u}_{n-1} < \bold{1}_1$ and $\bold{u}_{n+1}= \bold{1}_1$. 
			Hence \eqref{Hassan_mini_3_eq:1.11} implies that
			\begin{equation*}
			\medmuskip=-0.2mu
			\thinmuskip=-0.2mu
			\thickmuskip=-0.2mu
			\nulldelimiterspace=0.9pt
			\scriptspace=1.5pt    
			\arraycolsep1.5em
			\bold{w}_n=  T_1 \bold{u}_{n-1}+  T_2 \bold{1}_1= P \left(T_1 \bold{1}_1 +  T_2 P \, \bold{u}_{n-1}\right)=P \left( T_1 \bold{1}_1 +  T_2 \, \bold{u}_{N+2-n}\right)= P \, \bold{w}_{N+1-n} < \bold{1}_1.
			\end{equation*}
			
			\item Let $j=n+1$: By the induction hypothesis we have that $\bold{u}_{n}= P \, \bold{u}_{N+1-n}$, \linebreak
			$\bold{u}_{n+2}= P \, \bold{u}_{N-1-n}$, $\bold{u}_{n}< \bold{1}_1$ and $\bold{u}_{n+2}\leq \bold{1}_1$.
			Using \eqref{Hassan_mini_3_eq:1.9} and \eqref{Hassan_mini_3_eq:1.11} we get
			\begin{equation*}
			\medmuskip=-1mu
			\thinmuskip=-1mu
			\thickmuskip=-1mu
			\nulldelimiterspace=0.9pt
			\scriptspace=0.9pt    
			\arraycolsep0.9em
			\bold{w}_{n+1}=  T_1 \bold{u}_{n}+  T_2\bold{u}_{n+2}=P \left(T_1 P \, \bold{u}_{n+2} +  T_2 P \, \bold{u}_n \right)=P \left( T_1 \bold{u}_{N-1-n} +  T_2 \bold{u}_{N+1-n}\right)= P \, \bold{w}_{N-n} < \bold{1}_1.
			\end{equation*}
			It remains to show that $\bold{w}_k = \bold{1}_1$ for all $k \in \{n+2, \ldots, \ceil*{\frac{N}{2}}\} \cup \{\ceil*{\frac{N}{2}}, \ldots, N-1-n\}$. The induction hypothesis yields that $\bold{u} = T^n \bold{1}_N \in V^n$. Hence, $\bold{u}_k = \bold{1}_1$ for all $k \in \{n+1, \ldots, \ceil*{\frac{N}{2}}\} \cup \{\ceil*{\frac{N}{2}}, \ldots, N-n\}$. The result now follows by applying Equation \eqref{Hassan_mini_3_eq:1.11}.
			
		\end{enumerate}
	\end{proof}
	
	Lemma \ref{Hassan_mini_3_lem:4} implies that $T^{\ceil*{\frac{N}{2}}} \bold{1}_N \in V^{\ceil*{\frac{N}{2}}}$ so that $\Vert T^{\ceil*{\frac{N}{2}}} \Vert= \Vert T^{\ceil*{\frac{N}{2}}} \bold{1}_N\Vert < 1$, which is precisely Theorem 1 $(\rm{a})$.
	
	%%%%%%%%%%%%%%%%%%%%%%%%%%%%%%%%%%%%
	%%%%%%%%%%%%%%%%%%%%%%%%%%%%%%%%%%%%
	\section{Proof of Theorem 1 $(\rm{b})$}\label{Hassan_mini_3_subs:2}
	
	We first prove an intermediate lemma.
	
	\begin{lemma}\label{Hassan_mini_3_lem:4a}
		Let $n \in \left\{2, \ldots, \floor*{\frac{N}{2}}-2\right\}$, let $\bold{u}= T^n{\bold{1}}_N$ and $\bold{w}= T {\bold{u}}$. If for all $j \in \{1, \ldots, n\}$ it holds that
		\begin{align}\label{Hassan_mini_3_eq:hypothesis1}
		(\bold{u}_j)_1 \leq (\bold{u}_j)_2, \quad \text{and} \quad \bold{u}_j < \bold{u}_{j+1},
		\end{align}
		then for all  $j \in \{1, \ldots, n+1\}$ it holds that
		\begin{align}\label{Hassan_mini_3_eq:star2}
		(\bold{w}_j)_1 \leq (\bold{w}_j)_2, \quad \text{and} \quad \bold{w}_j < \bold{w}_{j+1}.
		\end{align}
	\end{lemma}
	\begin{proof}
		We prove the result by induction over the subdomain index $j$. The definition of the matrices $\widetilde{ T_1}$ and $\widetilde{ T_2}$ implies that $0=(\bold{w}_1)_1 \leq (\bold{w}_1)_2$, and Equation \eqref{Hassan_mini_3_eq:1.9} yields that 
		$$\bold{w}_1 = \begin{bmatrix} 0 \\ (1-\frac{\delta}{L})(\bold{u}_2)_1\end{bmatrix} < \begin{bmatrix} (1-\frac{\delta}{L})(\bold{u}_1)_2  \\ (1-\frac{\delta}{L})(\bold{u}_3)_1\end{bmatrix} \leq \begin{bmatrix} (1-\frac{\delta}{L})(\bold{u}_1)_2  +\frac{\delta}{L}(\bold{u}_3)_1\\ \frac{\delta}{L}(\bold{u}_1)_2+(1-\frac{\delta}{L})(\bold{u}_3)_1\end{bmatrix}=\bold{w}_2.$$ 
		
		We now proceed to the induction step. Assume that \eqref{Hassan_mini_3_eq:star2} holds for some $j \in \{1, \ldots, n\}$. We first show that $(\bold{w}_{j+1})_1 \leq (\bold{w}_{j+1})_2$. Equation \eqref{Hassan_mini_3_eq:1.9} implies that it is sufficient to show that $\bold{u}_j \leq \bold{u}_{j+2}$. There are two cases: $j<n-1$ and $j \in \{n-1, n\}$.
		If $j < n-1$, then \eqref{Hassan_mini_3_eq:hypothesis1} yields the required result.
		If $j \in \{n-1, n\}$, then \eqref{Hassan_mini_3_eq:hypothesis1} and the fact that $\bold{u} \in V^n$ gives that $\bold{u}_j < \bold{u}_{j+2}=\bold{1}_1$.
		
		Next, we show that $\bold{w}_{j+1} < \bold{w}_{j+2}$. Equation \eqref{Hassan_mini_3_eq:1.9} implies that it is sufficient to show that $\bold{u}_j < \bold{u}_{j+1}$ and $\bold{u}_{j+2} \leq \bold{u}_{j+3}$.
		There are three cases: $j < n-1$, $j=n-1$ and $j=n$.
		If $j < n-1$ then \eqref{Hassan_mini_3_eq:hypothesis1} yields the required result.
		If $j=n-1$ then \eqref{Hassan_mini_3_eq:hypothesis1} yields that $\bold{u}_j  < \bold{u}_{j+1}$ and the fact that $\bold{u} \in V^n$ gives that $\bold{u}_{j+2} =\bold{u}_{j+3}=\bold{1}_1$.
		If $j=n$ then \eqref{Hassan_mini_3_eq:hypothesis1} yields that $\bold{u}_j  < \bold{u}_{j+1}$ and it remains to show that $\bold{u}_{j+2} \leq\bold{u}_{j+3}$. To this end, we recall that $n \leq \floor*{\frac{N}{2}}-2$. Therefore, there are three sub-cases: 
		$n < \floor*{\frac{N}{2}}-2$ in which case $n+1<j+2, j+3 \leq \floor*{\frac{N}{2}}$; 
		$n= \floor*{\frac{N}{2}}-2$ and $N$ is even in which case $\bold{u}_{j+2} = P\bold{u}_{j+3}$;
		$n= \floor*{\frac{N}{2}}-2$ and $N$ is odd in which case $\bold{u}_{j+3} = P\bold{u}_{j+1}$.
		%\begin{itemize}\itemsep0em
		%\item[\textbf{a)}] $k < \floor*{\frac{N}{2}}-2$ in which case $j+2, j+3 \leq \floor*{\frac{N}{2}}$.\\
		%\item[\textbf{b)}] $k= \floor*{\frac{N}{2}}-2$ and $N$ is even in which case $\bold{w}_{j+2} = P\bold{w}_{j+3}$.\\
		%\item[\textbf{c)}] $k= \floor*{\frac{N}{2}}-2$ and $N$ is odd in which case $\bold{w}_{j+3} = P\bold{w}_{j+1}$.
		%\end{itemize}
		In all three sub-cases, we obtain that $\bold{u}_{j+2} =\bold{u}_{j+3}=\bold{1}_1$.
	\end{proof}

	Lemma \ref{Hassan_mini_3_lem:5} below describes the `shape' of the vector $T^n \bold{1}_N$ for natural numbers $n < \floor*{\frac{N}{2}}$. 
	
	\begin{lemma}\label{Hassan_mini_3_lem:5}
		Let $n \in \left\{1, \ldots, \floor*{\frac{N}{2}}-1\right\}$ be a natural number and let $\bold{w}= T^n {\bold{1}}_N$. Then for all $j \in \{1, \ldots, n\}$ it holds that
		\begin{align*}
		(\bold{w}_j)_1 \leq (\bold{w}_j)_2, &\quad (\bold{w}_{N+1-j})_2 \leq (\bold{w}_{N+1-j})_1,\; \text{ and} \quad
		\bold{w}_j < \bold{w}_{j+1}, \quad\bold{w}_{N+1-j} < \bold{w}_{N-j}.
		\end{align*}
	\end{lemma}
	\begin{proof}
		
		By Lemma \ref{Hassan_mini_3_lem:4} it holds that $\bold{w}_j = P\, \bold{w}_{N+1-j}$ for each $j \in \{1, \ldots, n\}$ so it suffices to show that for each $n \in \left\{1, \ldots, \floor*{\frac{N}{2}}-1\right\}$ and all $j \in \{1, \ldots, n\}$ it holds that 
		\begin{align}\label{Hassan_mini_3_eq:star3}
		(\bold{w}_j)_1 \leq (\bold{w}_j)_2 \text{ and }\bold{w}_j < \bold{w}_{j+1}.
		\end{align}
		We prove the result by induction over the iteration number $n$.
		Let $n=1$. The definition of the matrix $\widetilde{ T_2}$ and \eqref{Hassan_mini_3_eq:1.11} yield that $\bold{w}_1 = \begin{bmatrix} 0\\ 1-\frac{\delta}{L}\end{bmatrix}$ and $\bold{w}_2= \bold{1}_1$. 
		Thus, \eqref{Hassan_mini_3_eq:star3} holds for $n=1$. Next, let $n=2$ and let $\bold{u}= T^ {\bold{1}}_N$. The definition of the matrix $\widetilde{T_2}$ together with Equations \eqref{Hassan_mini_3_eq:1.7} and \eqref{Hassan_mini_3_eq:1.11} yields $\bold{w}_1 = \begin{bmatrix} 0\\ 1-\frac{\delta}{L}\end{bmatrix}$, $\bold{w}_1 <\bold{w}_2 < \bold{1}_1$ and $\bold{w}_3 =\bold{1}_1$. Thus, \eqref{Hassan_mini_3_eq:star3} holds for $n=2$. 
		Finally, assume that \eqref{Hassan_mini_3_eq:star3} holds for some $n \in \{2, \ldots, \floor*{\frac{N}{2}}-2\}$. It follows from Lemma \ref{Hassan_mini_3_lem:4a} that \eqref{Hassan_mini_3_eq:star3} also holds for $n+1$.
	\end{proof}
	
	Next, Lemma \ref{Hassan_mini_3_cor:2} describes the `shape' of the vector $T^n \bold{1}_N$ for natural numbers $n \geq \floor*{\frac{N}{2}}$. Together, Lemmas \ref{Hassan_mini_3_lem:5} and \ref{Hassan_mini_3_cor:2} establish that the vector $T^n \bold{1}_N$ is monotonically increasing as one moves from the extrema of $\Omega$ towards its centre.
	
	\begin{lemma}\label{Hassan_mini_3_cor:2}
		Let $n\geq \floor*{\frac{N}{2}}$ be a natural number, and let $\bold{w}=  T^{n}{\bold{1}}_N$. Then for all $j \in \{1, \ldots, \floor*{\frac{N}{2}}-1\}$ it holds that
		\begin{align*}
		(\bold{w}_j)_1 \leq (\bold{w}_j)_2, \quad (\bold{w}_{N+1-j})_2 \leq (\bold{w}_{N+1-j})_1, \quad \text{and} \quad \bold{w}_j < \bold{w}_{j+1}, \quad\bold{w}_{N+1-j} < \bold{w}_{N-j}.
		\end{align*}
	\end{lemma}
	In addition, if $N$ is an odd number, then $\bold{w}_{\floor*{\frac{N}{2}}} \leq \bold{w}_{\ceil*{\frac{N}{2}}}$ and $\bold{w}_{\floor*{\frac{N}{2}}+2} < \bold{w}_{\ceil*{\frac{N}{2}}}$.
	\begin{proof}
		Lemma \ref{Hassan_mini_3_cor:2} can be proven in a similar manner to Lemma \ref{Hassan_mini_3_lem:5} using a proof-by-induction on the iteration number $n$. We omit it here for brevity.
		%In order to avoid tedious repetition, we desist from providing a step-by-step proof. 
	\end{proof}
	
	We are now ready to prove our second main result.
	
	\begin{proof}[Theorem \ref{Hassan_mini_3_thm:Main} ${\rm(\rm b)}$]
		Assume that $N \in \mathbb{N}$ is even. In view of Lemma \ref{Hassan_mini_3_lemma:1}, we must prove that $\Vert  T^{n+1}\bold{1}_N\Vert < \Vert  T^n {\bold{1}}_N\Vert$. Let $\bold{w}= T^{n+1}{\bold{1}}_N$ and $\bold{u}= T^n {\bold{1}}_N$. By Lemma \ref{Hassan_mini_3_cor:2}, we know that $\Vert  T^{n+1}{\bold{1}}_N\Vert = \Vert \bold{w}_{\ceil*{\frac{N}{2}}}\Vert$ and $\Vert  T^{n}{\bold{1}}_N\Vert = \Vert \bold{u}_{\ceil*{\frac{N}{2}}}\Vert$. 
		
		Since $\bold{w}=T \bold{u}$, we have that $\bold{w}_{\ceil*{\frac{N}{2}}} =  T_1 \bold{u}_{\ceil*{\frac{N}{2}}-1} +  T_2\bold{u}_{\ceil*{\frac{N}{2}}+1}$. Since $N$ is even, we obtain that $\ceil*{\frac{N}{2}}+1= \frac{N}{2}+1$ and thus, Lemma \ref{Hassan_mini_3_lem:4} yields that $\bold{u}_{\ceil*{\frac{N}{2}}+1}=P\bold{u}_{\ceil*{\frac{N}{2}}}$. It follows that $\bold{w}_{\ceil*{\frac{N}{2}}} =  T_1 \bold{u}_{\ceil*{\frac{N}{2}}-1} +  T_2P\bold{u}_{\ceil*{\frac{N}{2}}}$. From \eqref{Hassan_mini_3_eq:1.9} we also obtain that 
		\begin{align}\label{Hassan_mini_3_eq:newstar}
		\bold{u}_{\ceil*{\frac{N}{2}}-1}\leq \bold{w}_{\ceil*{\frac{N}{2}}} \leq P \, \bold{u}_{\ceil*{\frac{N}{2}}},
		\end{align}
		where the equality holds if and only if $(\bold{u}_{\ceil*{\frac{N}{2}}-1})_2= (P \, \bold{u}_{\ceil*{\frac{N}{2}}})_1=(\bold{u}_{\ceil*{\frac{N}{2}}})_2$. We know from Lemma \ref{Hassan_mini_3_cor:2} that $\bold{u}_{\ceil*{\frac{N}{2}}-1}< \bold{u}_{\ceil*{\frac{N}{2}}}$, which yields that $(\bold{u}_{\ceil*{\frac{N}{2}}-1})_2<(\bold{u}_{\ceil*{\frac{N}{2}}})_2$. Hence the inequalities in \eqref{Hassan_mini_3_eq:newstar} are strict:
		$\bold{u}_{\ceil*{\frac{N}{2}}-1}< \bold{w}_{\ceil*{\frac{N}{2}}} < P \, \bold{u}_{\ceil*{\frac{N}{2}}}$.
		Hence, we obtain that $\Vert  T^{n+1}{\bold{1}}_N\Vert=\Vert \bold{w}_{\ceil*{\frac{N}{2}}}\Vert < \Vert P\,\bold{u}_{\ceil*{\frac{N}{2}}}\Vert=\Vert \bold{u}_{\ceil*{\frac{N}{2}}}\Vert=\Vert  T^{n}{\bold{1}}_N\Vert$. This completes the proof of the first assertion.
		
		Assume now that $N \in \mathbb{N}$ is odd. Let $\bold{u}= T^n {\bold{1}}_N$, $\bold{w}= T^{n+1}{\bold{1}}_N$ and $\bold{y}=  T^{n+2}{\bold{1}}_N$. Lemma \ref{Hassan_mini_3_cor:2} implies that
		\begin{align*}
		\Vert  T^{n+2}{\bold{1}}_N\Vert &= \Vert \bold{y}_{\ceil*{\frac{N}{2}}}\Vert, \qquad \Vert  T^{n+1}{\bold{1}}_N\Vert = \Vert \bold{w}_{\ceil*{\frac{N}{2}}}\Vert, \qquad \Vert  T^{n}{\bold{1}}_N\Vert = \Vert \bold{u}_{\ceil*{\frac{N}{2}}}\Vert.
		\end{align*}
		Since $\Vert  T \Vert=1$, we have
		\begin{align}\label{Hassan_mini_3_eq:final1}
		\Vert \bold{y}_{\ceil*{\frac{N}{2}}}\Vert \leq \Vert \bold{w}_{\ceil*{\frac{N}{2}}}\Vert \leq \Vert \bold{u}_{\ceil*{\frac{N}{2}}}\Vert.
		\end{align}
		Clearly if $\Vert \bold{w}_{\ceil*{\frac{N}{2}}}\Vert < \Vert \bold{u}_{\ceil*{\frac{N}{2}}}\Vert$ then \eqref{Hassan_mini_3_eq:final1} yields that
		\begin{align*}
		\Vert  T^{n+2}{\bold{1}}_N\Vert \leq \Vert  T^{n+1}{\bold{1}}_N\Vert=\Vert \bold{w}_{\ceil*{\frac{N}{2}}}\Vert < \Vert \bold{u}_{\ceil*{\frac{N}{2}}}\Vert= \Vert  T^{n}{\bold{1}}_N\Vert, 
		\end{align*}
		which is our claim. Suppose that $\Vert \bold{w}_{\ceil*{\frac{N}{2}}}\Vert = \Vert \bold{u}_{\ceil*{\frac{N}{2}}}\Vert$. We show that $\Vert \bold{y}_{\ceil*{\frac{N}{2}}}\Vert < \Vert \bold{w}_{\ceil*{\frac{N}{2}}}\Vert$ which implies our claim. To do so, since $N$ is odd, we use Lemma \ref{Hassan_mini_3_lem:4} together with the facts that $\bold{y}=T\bold{w}$ and $\bold{w}=T\bold{u}$ to obtain that
		$\bold{y}_{\ceil*{\frac{N}{2}}}=  T_1 \bold{w}_{\ceil*{\frac{N}{2}}-1}+ T_2 P\,\bold{w}_{\ceil*{\frac{N}{2}}-1}$
		and $\bold{w}_{\ceil*{\frac{N}{2}}}=  T_1 \bold{u}_{\ceil*{\frac{N}{2}}-1}+ T_2 P\,\bold{u}_{\ceil*{\frac{N}{2}}-1}$
		which implies
		\begin{align} \label{Hassan_mini_3_eq:square}
		\bold{y}_{\ceil*{\frac{N}{2}}} &= (\bold{w}_{\ceil*{\frac{N}{2}}-1})_2 \bold{1}_1,\qquad
		\bold{w}_{\ceil*{\frac{N}{2}}} = (\bold{u}_{\ceil*{\frac{N}{2}}-1})_2 \bold{1}_1.
		\end{align}
		From Lemma \ref{Hassan_mini_3_lem:4}, we know that 
		\begin{align}\label{Hassan_mini_3_eq:circle}
		\bold{u}_{\ceil*{\frac{N}{2}}}\stackrel{\text{Lemma \ref{Hassan_mini_3_lem:4}}}{=} \bold{w}_{\ceil*{\frac{N}{2}}}\stackrel{\eqref{Hassan_mini_3_eq:square}}{=} (\bold{u}_{\ceil*{\frac{N}{2}}-1})_2 \bold{1}_1. 
		\end{align}
		Using the fact that $\bold{w}=T \bold{u}$ and Equation \eqref{Hassan_mini_3_eq:circle} we have that
		\begin{align*}
		\bold{w}_{\ceil*{\frac{N}{2}}-1} \stackrel{\eqref{Hassan_mini_3_eq:1.9}}{=}  T_1 \bold{u}_{\ceil*{\frac{N}{2}}-2}+ T_2 \,\bold{u}_{\ceil*{\frac{N}{2}}} \stackrel{\eqref{Hassan_mini_3_eq:circle}}{=}  T_1 \bold{u}_{\ceil*{\frac{N}{2}}-2}+ T_2 (\bold{u}_{\ceil*{\frac{N}{2}}-1})_2 \bold{1}_1. 
		\end{align*}
		Using \eqref{Hassan_mini_3_eq:1.9} we obtain that
		\begin{align}\label{Hassan_mini_3_eq:triangle2}
		\bold{u}_{\ceil*{\frac{N}{2}}-2} \leq \bold{w}_{\ceil*{\frac{N}{2}}-1} \leq (\bold{u}_{\ceil*{\frac{N}{2}}-1})_2 \bold{1}_1.
		\end{align}
		where the equality holds if and only if $(\bold{u}_{\ceil*{\frac{N}{2}}-2})_2=(\bold{u}_{\ceil*{\frac{N}{2}}-1})_2$. However, Lemma \ref{Hassan_mini_3_cor:2} implies that $\bold{u}_{\ceil*{\frac{N}{2}}-2} < \bold{u}_{\ceil*{\frac{N}{2}}-1}$ which immediately yields that $(\bold{u}_{\ceil*{\frac{N}{2}}-2})_2<(\bold{u}_{\ceil*{\frac{N}{2}}-1})_2$. Hence the inequalities in \eqref{Hassan_mini_3_eq:triangle2} are strict and thus
		\begin{align}\label{Hassan_mini_3_eq:hexagon}
		\bold{w}_{\ceil*{\frac{N}{2}}-1} \stackrel{\eqref{Hassan_mini_3_eq:triangle2}}{<} (\bold{u}_{\ceil*{\frac{N}{2}}-1})_2 \bold{1}_1 \stackrel{\eqref{Hassan_mini_3_eq:square}}{=} \bold{w}_{\ceil*{\frac{N}{2}}}.
		\end{align}
		Recalling \eqref{Hassan_mini_3_eq:square} we obtain that
		$\bold{y}_{\ceil*{\frac{N}{2}}} \stackrel{\eqref{Hassan_mini_3_eq:square}}{=}(\bold{w}_{\ceil*{\frac{N}{2}}-1})_2 \bold{1}_1 \stackrel{\eqref{Hassan_mini_3_eq:hexagon}}{<} \bold{w}_{\ceil*{\frac{N}{2}}}$,
		which completes the proof.
	\end{proof}
	
	\bibliographystyle{plain}
	\bibliography{Hassan_mini_3}
\end{document}